\documentclass[12pt]{article}

\usepackage{amsmath,amsthm,amssymb}
\usepackage[top=30truemm,bottom=30truemm,left=25truemm,right=25truemm]{geometry}
\usepackage[dvips]{graphicx,color,psfrag}
\usepackage{bm}

\theoremstyle{plain}
\newtheorem{definition}{Definition}[section]
\newtheorem{thm}[definition]{Theorem}

\newtheorem{lem}[definition]{Lemma}
\newtheorem{cor}[definition]{Corollary}

\newtheorem{rem}[definition]{Remark}

\title{Integer group determinants for ${\rm C}_{4} \rtimes {\rm C}_{4}$}
\author{Yuka Yamaguchi and Naoya Yamaguchi}
\date{\today}

\begin{document}

\maketitle

\begin{abstract}
Let ${\rm C}_{4}$ be the cyclic group of order $4$. 
We determine all possible values of the integer group determinant of ${\rm C}_{4} \rtimes {\rm C}_{4}$. 
\end{abstract}

\section{Introduction}
For a finite group $G$, 
let $\Theta_{G}(x_{g})$ be the group determinant of $G$ and let 
$$
S(G) := \left\{ \Theta_{G}(x_{g}) \mid x_{g} \in \mathbb{Z} \right\}. 
$$
In this paper,  we determine $S ({\rm C}_{4} \rtimes {\rm C}_{4})$, where
\begin{align*}
{\rm C}_{4} \rtimes {\rm C}_{4} := \langle g_{1}, g_{2} \mid g_{1}^{4} = g_{2}^{4} = e, \: g_{2} g_{1} = g_{1}^{3} g_{2} \rangle = \{ g_{1}^{s} g_{2}^{t} \mid 0 \leq s, t \leq 3 \}. 
\end{align*}
Here, $e$ denotes the unit element. 

\begin{thm}\label{thm:1.1}
Let 
\begin{align*}
P &:= \left\{ p \mid p \equiv - 3 \: ({\rm mod} \: 8) \: \text{is a prime number} \right\}, \\ 
P' &:= \left\{ p \mid p \equiv - 1 \: \: ({\rm mod} \: 4) \: \: \text{is a prime number} \right\}, \\ 
A &:= \left\{ (8 k - 3) (8 l - 3) (4 m - 1)^{2} \mid k, m \in \mathbb{Z}, \: m \geq 1, \: 8 l - 3 \in P, \: k + l \equiv m \: \: ({\rm mod} \: 2) \right\} \\ 
&\subsetneq \{ 16 m - 7 \mid m \in \mathbb{Z} \}. 
\end{align*}
Then we have 
\begin{align*}
S({\rm C}_{4} \rtimes {\rm C}_{4}) = \left\{ 16 m + 1, \: 2^{14} p (2m + 1), \: 2^{14} q^{2} (2 m + 1), \: 2^{15} m \mid m \in \mathbb{Z}, \: p \in P, \: q \in P' \right\} \cup A. 
\end{align*}
\end{thm}

We recall the definition of the group determinant. 
For a finite group $G$, 
let $x_{g}$ be a variable for each $g \in G$ and let $\mathbb{Z}[x_{g}]$ be the multivariate polynomial ring in $x_{g}$ over $\mathbb{Z}$. 
The group determinant $\Theta_{G}(x_{g})$ of $G$ was defined by Dedekind as follows: 
$$
\Theta_{G}(x_{g}) := \det{\left( x_{g h^{- 1}} \right)_{g, h \in G}} \in \mathbb{Z}[x_{g}]. 
$$
The irreducible factorization of $\Theta_{G}(x_{g})$ over $\mathbb{C}$ was given by Frobenius \cite{Frobenius1968gruppen}: 
{\it Let $G$ be a finite group and let $\widehat{G}$ be a complete set of representatives of the equivalence classes of irreducible representations of $G$ over $\mathbb{C}$. Then 
$$
\Theta_{G}(x_{g}) = \prod_{\varphi \in \widehat{G}} \det{\left( \sum_{g \in G} \varphi(g) x_{g} \right)^{\deg{\varphi}}}, 
$$
where $\deg{\varphi}$ is the degree of $\varphi$.}
This is called Frobenius' theorem. 

A group determinant called an integer group determinant when its variables are integers. 
At the meeting of the American Mathematical Society in Hayward, California, in April 1977, 
Olga Taussky-Todd \cite{Olga} asked whether one could characterize the values of the integer group determinant \cite[p.1]{https://doi.org/10.48550/arxiv.2302.11688}. 
For some groups, $S(G)$ was determined in \cite{MR3879399, MR2914452, MR624127, https://doi.org/10.48550/arxiv.2205.12439, MR601702, MR550657, MR4363104, MR4313424, https://doi.org/10.48550/arxiv.2211.09930, https://doi.org/10.48550/arxiv.2302.11688, MR3998922, MR4056860, https://doi.org/10.48550/arxiv.2203.14420, Yamaguchi, https://doi.org/10.48550/arxiv.2211.14761, https://doi.org/10.48550/arxiv.2209.12446, https://doi.org/10.48550/arxiv.2211.01597, arXiv:2303.08489}. 
As a result, for every group $G$ of order at most $15$, a complete description was obtained for $S(G)$. 
Also, for all five abelian groups of order $16$, and the non-abelian groups ${\rm D}_{16}, {\rm D}_{8} \times {\rm C}_{2}, {\rm Q}_{8} \times {\rm C}_{2}, {\rm Q}_{16}$, ${\rm C}_{2}^{2} \rtimes {\rm C}_{4}$, 
$S(G)$ have been determined. 
There are fourteen groups of order $16$ up to isomorphism \cite{MR1510814, MR1505615}. 
Theorem~$\ref{thm:1.1}$ leaves three unresolved groups of order $16$.

\section{Preliminaries}
For any $\overline{r} \in {\rm C}_{4}$ with $r \in \{ 0, 1, 2, 3 \}$, 
we denote the variable $x_{\overline{r}}$ by $x_{r}$, 
and let 
$$
D_{4}(x_{0}, x_{1}, x_{2}, x_{3}) := \det{\left( x_{g h^{- 1}} \right)}_{g, h \in {\rm C}_{4}}. 
$$
For any $(\overline{r}, \overline{s}) \in {\rm C}_{4} \times {\rm C}_{2}$ with $r \in \{ 0, 1, 2, 3 \}$ and $s \in \{ 0, 1 \}$, 
we denote the variable $y_{(\overline{r}, \overline{s})}$ by $y_{j}$, where $j := r + 4 s$, 
and let 
$$
D_{4 \times 2}(y_{0}, y_{1}, \ldots, y_{7}) := \det{\left( y_{g h^{- 1}} \right)}_{g, h \in {\rm C}_{4} \times {\rm C}_{2}}. 
$$
From the $H = {\rm C}_{4}$ and $K = {\rm C}_{2}$ case of \cite[Theorem~1.1]{https://doi.org/10.48550/arxiv.2202.06952}, 
we have the following corollary. 

\begin{cor}\label{cor:2.1}
We have 
$$
D_{4 \times 2}(y_{0}, y_{1}, \ldots, y_{7}) = D_{4}(y_{0} + y_{4}, y_{1} + y_{5}, y_{2} + y_{6}, y_{3} + y_{7}) D_{4}(y_{0} - y_{4}, y_{1} - y_{5}, y_{2} - y_{6}, y_{3} - y_{7}). 
$$
\end{cor}

Let $G := C_{4} \rtimes C_{4}$. 
For any $g = g_{1}^{s} g_{2}^{t} \in G \: (0 \leq s, t \leq 3)$, 
we denote the variable $z_{g}$ by $z_{j}$, where $j := t + 4 s$, 
and let 
$$
D_{G}(z_{0}, z_{1}, \ldots, z_{15}) := \det{(z_{g h^{- 1}})_{g, h \in G}}. 
$$
Let $[x] := \max{ \{ n \in \mathbb{Z} \mid n \leq x \} }$ and $\widehat{G} := \left\{ \varphi_{k} \mid 0 \leq k \leq 9 \right\}$, 
where 
\begin{align*}
\varphi_{k}(g_{1}) := (- 1)^{[ \frac{k}{4} ]}, \quad \varphi_{k}(g_{2}) := \sqrt{- 1}^{k}
\end{align*}
for any $0 \leq k \leq 7$ and 
\begin{align*}
\varphi_{k}(g_{1}) := (- 1)^{k} \begin{pmatrix} 0 & - 1 \\ 1 & 0 \end{pmatrix}, \quad
\varphi_{k}(g_{2}) := \sqrt{- 1}^{k} \begin{pmatrix} 1 & 0 \\ 0 & - 1 \end{pmatrix} 
\end{align*}
for any $8 \leq k \leq 9$. 
Then, $\widehat{G}$ is a complete set of representatives of the equivalence classes of irreducible representations of $G$ over $\mathbb{C}$. 
Let $F$ be the eight variable polynomial defined by 
$$
F(w_{0}, w_{1}, \ldots, w_{7}) := f_{0}(w_{0} - w_{2}, w_{4} - w_{6}, w_{1} - w_{3}, w_{5} - w_{7}) f_{1}(w_{0} + w_{2}, w_{4} + w_{6}, w_{1} + w_{3}, w_{5} + w_{7}),
$$
where $f_{k}(x, y, z, w) := x^{2} + y^{2} + (- 1)^{k} z^{2} + (- 1)^{k} w^{2} \: (0 \leq k \leq 1)$. 
Then we have the following. 

\begin{lem}\label{lem:2.2}
We have 
$$
D_{G}(z_{0}, z_{1}, \ldots, z_{15}) = D_{4 \times 2}(z_{0} + z_{8}, z_{1} + z_{9}, \ldots, z_{7} + z_{15}) F(z_{0} - z_{8}, z_{1} - z_{9}, \ldots, z_{7} - z_{15})^{2}. 
$$
\end{lem}
\begin{proof}
From Frobenius' theorem, we have 
\begin{align*}
D_{G}(z_{0}, z_{1}, \ldots, z_{15}) = \prod_{\varphi \in \widehat{G}} \det{\left( \sum_{g \in G} \varphi(g) z_{g} \right)}^{\deg{\varphi}} = \prod_{k = 0}^{9} \det{M_{k}}^{\deg{\varphi_{k}}}, 
\end{align*}
where $M_{k} := \sum_{s = 0}^{3} \sum_{t = 0}^{3} \varphi_{k}(g_{1}^{s} g_{2}^{t}) z_{t + 4 s}$. 
Note that $\{ \chi_{k} \mid 0 \leq k \leq 7 \}$ is the dual group of ${\rm C}_{4} \times {\rm C}_{2}$, where 
$\chi_{k}( (\overline{1}, \overline{0}) ) := \sqrt{- 1}^{k}$ and $\chi_{k}( (\overline{0}, \overline{1}) ) := (- 1)^{[\frac{k}{4}]}$ for $0 \leq k \leq 7$. 
Then we have 
\begin{align*}
\prod_{k = 0}^{7} \det{M_{k}}^{\deg{\varphi_{k}}} 
&= \prod_{k = 0}^{7} M_{k} \\ 
&= \prod_{k = 0}^{7} \sum_{s = 0}^{3} \sum_{t = 0}^{3} (- 1)^{[\frac{k}{4}] s} \sqrt{- 1}^{k t} z_{t + 4 s} \\ 
&= \prod_{k = 0}^{7} \sum_{t = 0}^{3} \sum_{s = 0}^{1} \sqrt{- 1}^{k t} (- 1)^{[\frac{k}{4}] s} (z_{t + 4 s} + z_{t + 4 (s + 2)}) \\
&= \prod_{k = 0}^{7} \sum_{t = 0}^{3} \sum_{s = 0}^{1} \chi_{k}( (\overline{t}, \overline{s}) ) (z_{t + 4 s} + z_{t + 4 s + 8}) \\ 
&= D_{4 \times 2}(z_{0} + z_{8}, z_{1} + z_{9}, \ldots, z_{7} + z_{15}). 
\end{align*}
Also, for any $8 \leq k \leq 9$, we have 
\begin{align*}
M_{k} 
&= \sum_{s = 0}^{3} \sum_{t = 0}^{3} (- 1)^{k s} \sqrt{- 1}^{k t} z_{t + 4 s} 
\begin{pmatrix} 0 & - 1 \\ 1 & 0 \end{pmatrix}^{s} \begin{pmatrix} 1 & 0 \\ 0 & - 1 \end{pmatrix}^{t} \\ 
&= \sum_{t = 0}^{3} \sum_{s = 0}^{1} \sqrt{- 1}^{k t} (- 1)^{k s} (z_{t + 4 s} - z_{t + 4 s + 8}) 
\begin{pmatrix} 0 & - 1 \\ 1 & 0 \end{pmatrix}^{s} \begin{pmatrix} 1 & 0 \\ 0 & - 1 \end{pmatrix}^{t} \\ 
&= \sum_{t = 0}^{3} \left\{ \sqrt{- 1}^{k t} (z_{t} - z_{t + 8}) \begin{pmatrix} 1 & 0 \\ 0 & (- 1)^{t} \end{pmatrix} 
+ \sqrt{- 1}^{k t} (- 1)^{k} (z_{t + 4} - z_{t + 12}) \begin{pmatrix} 0 & (- 1)^{t + 1} \\ 1 & 0 \end{pmatrix} \right\} \\ 
&= 
\begin{pmatrix} 
\sum_{t = 0}^{3} \sqrt{- 1}^{k t} \tilde{z}_{t} & \sum_{t = 0}^{3} (- 1)^{k + t + 1} \sqrt{- 1}^{k t} \tilde{z}_{t + 4} \\ 
\sum_{t = 0}^{3} (- 1)^{k} \sqrt{- 1}^{k t} \tilde{z}_{t + 4} & \sum_{t = 0}^{3} (- 1)^{t} \sqrt{- 1}^{k t} \tilde{z}_{t} 
\end{pmatrix}, 
\end{align*}
where $\tilde{z}_{t} := z_{t} - z_{t + 8}$ for $0 \leq t \leq 7$. 
Therefore, 
\begin{align*}
\det{M_{k}} 
&= \sum_{s = 0}^{3} \sqrt{- 1}^{k s} \tilde{z}_{s} \sum_{t = 0}^{3} (- 1)^{t} \sqrt{- 1}^{k t} \tilde{z}_{t} 
+ \sum_{s = 0}^{3} \sqrt{- 1}^{k s} \tilde{z}_{s + 4} \sum_{t = 0}^{3} (- 1)^{t} \sqrt{- 1}^{k t} \tilde{z}_{t + 4} \\ 
&= \left\{ \tilde{z}_{0} + (- 1)^{k} \tilde{z}_{2} \right\}^{2} - (- 1)^{k} \left\{ \tilde{z}_{1} + (- 1)^{k} \tilde{z}_{3} \right\}^{2} \\ 
&\qquad + \left\{ \tilde{z}_{4} + (- 1)^{k} \tilde{z}_{6} \right\}^{2} - (- 1)^{k} \left\{ \tilde{z}_{5} + (- 1)^{k} \tilde{z}_{7} \right\}^{2}. 
\end{align*}
That is, 
\begin{align*}
\det{M_{8}} = f_{1}(\tilde{z}_{0} + \tilde{z}_{2}, \tilde{z}_{4} + \tilde{z}_{6}, \tilde{z}_{1} + \tilde{z}_{3}, \tilde{z}_{5} + \tilde{z}_{7}), \quad 
\det{M_{9}} = f_{0}(\tilde{z}_{0} - \tilde{z}_{2}, \tilde{z}_{4} - \tilde{z}_{6}, \tilde{z}_{1} - \tilde{z}_{3}, \tilde{z}_{5} - \tilde{z}_{7}). 
\end{align*} 
This completes the proof. 
\end{proof}

Throughout this paper, 
we assume that $a_{0}, a_{1}, \ldots, a_{15} \in \mathbb{Z}$ and let 
\begin{align*}
&&b_{i} &:= (a_{i} + a_{i + 8}) + (a_{i + 4} + a_{i + 12}) && (0 \leq i \leq 3), \\
&&c_{i} &:= (a_{i} + a_{i + 8}) - (a_{i + 4} + a_{i + 12}) && (0 \leq i \leq 3), \\
&&d_{i} &:= a_{i} - a_{i + 8} && (0 \leq i \leq 7). &
\end{align*}
Also, let 
\begin{align*}
\bm{a} := (a_{0}, a_{1}, \ldots, a_{15}), \quad 
\bm{b} := (b_{0}, b_{1}, b_{2}, b_{3}), \quad 
\bm{c} := (c_{0}, c_{1}, c_{2}, c_{3}), \quad 
\bm{d} := (d_{0}, d_{1}, \ldots, d_{7}). 
\end{align*}
Then, from Corollary~$\ref{cor:2.1}$ and Lemma~$\ref{lem:2.2}$, 
we have 
\begin{align*}
D_{G}(\bm{a}) 
&= D_{4 \times 2}(a_{0} + a_{8}, a_{1} + a_{9}, \ldots, a_{7} + a_{15}) F(a_{0} - a_{8}, a_{1} - a_{9}, \ldots, a_{7} - a_{15})^{2} \\ 
&= D_{4}(\bm{b}) D_{4}(\bm{c}) F(\bm{d})^{2}. 
\end{align*}

\begin{rem}\label{rem:2.3}
For any $0 \leq i \leq 3$, we have $b_{i} \equiv c_{i} \equiv d_{i} + d_{i + 4} \pmod{2}$. 
\end{rem}

\begin{lem}\label{lem:2.4}
The following hold: 
\begin{enumerate}
\item[$(1)$] $D_{4}(\bm{b}) \equiv b_{0} + b_{1} + b_{2} + b_{3} \pmod{2}$; 
\item[$(2)$] $D_{4}(\bm{c}) \equiv c_{0} + c_{1} + c_{2} + c_{3} \pmod{2}$; 
\item[$(3)$] $F(\bm{d}) \equiv d_{0} + d_{1} + \cdots + d_{7} \pmod{2}$. 
\end{enumerate}
\end{lem}
\begin{proof}
We obtain (1) and (2) from the following: for any $x_{0}, x_{1}, x_{2}, x_{3} \in \mathbb{Z}$, 
\begin{align*}
D_{4}(x_{0}, x_{1}, x_{2}, x_{3}) 
&= \left\{ (x_{0} + x_{2})^{2} - (x_{1} + x_{3})^{2} \right\} \left\{ (x_{0} - x_{2})^{2} + (x_{1} - x_{3})^{2} \right\} \\ 
&\equiv (x_{0}^{2} + x_{2}^{2} + x_{1}^{2} + x_{3}^{2})^{2} \\ 
&\equiv x_{0} + x_{1} + x_{2} + x_{3} \pmod{2}. 
\end{align*}
We prove (3). 
Since $f_{k}(x, y, z, w) = x^{2} + y^{2} + (- 1)^{k} z^{2} + (- 1)^{k} w^{2} \equiv x + y + z + w \pmod{2}$ for any $0 \leq k \leq 1$ and $x, y, z, w \in \mathbb{Z}$, 
we have 
\begin{align*}
F(\bm{d}) 
&= f_{0}(d_{0} - d_{2}, d_{4} - d_{6}, d_{1} - d_{3}, d_{5} - d_{7}) f_{1}(d_{0} + d_{2}, d_{4} + d_{6}, d_{1} + d_{3}, d_{5} + d_{7}) \\ 
&\equiv (d_{0} + d_{1} + \cdots + d_{7})^{2} \\ 
&\equiv d_{0} + d_{1} + \cdots + d_{7} \pmod{2}. 
\end{align*}
\end{proof}

From Remark~$\ref{rem:2.3}$ and Lemma~$\ref{lem:2.4}$, 
we have the following lemma. 

\begin{lem}\label{lem:2.5}
We have $D_{G}(\bm{a}) \equiv D_{4}(\bm{b}) \equiv D_{4}(\bm{c}) \equiv F(\bm{d}) \pmod{2}$. 
\end{lem}

\section{Impossible odd numbers}\label{section3}
In this section, 
we consider impossible odd numbers. 
Let $\mathbb{Z}_{\rm odd}$ be the set of all odd numbers. 
Let 
$P := \left\{ p \mid p \equiv - 3 \: ({\rm mod} \: 8) \: \text{is a prime number} \right\}$ and let 
\begin{align*}
A := \left\{ (8 k - 3) (8 l - 3) (4 m - 1)^{2} \mid k, m \in \mathbb{Z}, \: m \geq 1, \: 8 l - 3 \in P, \: k + l \equiv m \: \: ({\rm mod} \: {2}) \right\}. 
\end{align*}

\begin{lem}\label{lem:3.1}
We have $S({\rm C}_{4} \rtimes {\rm C}_{4}) \cap \mathbb{Z}_{\rm odd} \subset \{ 16 m + 1 \mid m \in \mathbb{Z} \} \cup A$. 
\end{lem}

Let 
\begin{align*}
n_{0} &:= f_{0}(d_{0} - d_{2}, d_{4} - d_{6}, d_{1} - d_{3}, d_{5} - d_{7}), & 
n_{1} &:= f_{1}(d_{0} + d_{2}, d_{4} + d_{6}, d_{1} + d_{3}, d_{5} + d_{7}), \\ 
d &:= (d_{0} + d_{2}) (d_{4} + d_{6}) + (d_{1} + d_{3}) (d_{5} + d_{7}), & 
d^{*} &:= d_{0} d_{2} + d_{4} d_{6} + d_{1} d_{3} + d_{5} d_{7}. 
\end{align*}

To prove Lemma~$\ref{lem:3.1}$, we use the following two remarks and four lemmas. 

\begin{rem}\label{rem:3.2}
From {\rm \cite[Lemma~2.10 (1)]{https://doi.org/10.48550/arxiv.2211.01597}}, 
we have $b_{0} b_{2} + b_{1} b_{3} + c_{0} c_{2} + c_{1} c_{3} \equiv 2 d^{*} \pmod{4}$. 
\end{rem}

\begin{rem}\label{rem:3.3}
Let $\beta := (d_{0} + d_{2} + \sqrt{- 1} d_{4} + \sqrt{- 1} d_{6})^{2} - (d_{1} + d_{3} + \sqrt{- 1} d_{5} + \sqrt{- 1} d_{7})^{2}$. 
Suppose that $b_{0} + b_{2} \not\equiv b_{1} + b_{3} \pmod{2}$. 
Then, $d_{0} + d_{2} + d_{4} + d_{6} \not\equiv d_{1} + d_{3} + d_{5} + d_{7} \pmod{2}$. 
Thus, from {\rm \cite[Lemma~2.8 (1)]{https://doi.org/10.48550/arxiv.2211.01597}}, we have 
\begin{align*}
\beta \overline{\beta} \equiv - 3 \pmod{8} 
&\iff \text{Exactly one of} \: \: d_{0} + d_{2}, d_{4} + d_{6}, d_{1} + d_{3}, d_{5} + d_{7} \: \: \text{is even} \\ 
&\iff d = (d_{0} + d_{2}) (d_{4} + d_{6}) + (d_{1} + d_{3}) (d_{5} + d_{7}) \equiv 1 \pmod{2}. 
\end{align*}
Therefore, the assumption $\beta \overline{\beta} \equiv - 3 \pmod{8}$ in {\rm \cite[Lemma~4.4]{https://doi.org/10.48550/arxiv.2211.01597}} can be replaced by $d \equiv 1 \pmod{2}$. 
\end{rem}

\begin{lem}\label{lem:3.4}
If $b_{0} + b_{2} \not\equiv b_{1} + b_{3} \pmod{2}$, 
then $D_{4}(\bm{b}) D_{4}(\bm{c}) \equiv 1 - 8 d^{*} \pmod{16}$. 
\end{lem}

Lemma~$\ref{lem:3.4}$ is immediately obtained from \cite[Lemma~4.3]{https://doi.org/10.48550/arxiv.2211.01597} and Remark~$\ref{rem:3.2}$. 

\begin{lem}\label{lem:3.5}
Suppose that $b_{0} + b_{2} \not\equiv b_{1} + b_{3}$, $d \equiv 1 \pmod{2}$. 
Then the following hold: 
\begin{enumerate}
\item[$(1)$] If $d^{*} \equiv 0 \pmod{2}$, then 
$$
D_{4}(\bm{b}) D_{4}(\bm{c}) \in \left\{ (8 k - 3) (8 l - 3) \mid k \in \mathbb{Z}, \: 8 l - 3 \in P, \: k \not\equiv l \: \: ({\rm mod} \: 2) \right\}; 
$$
\item[$(2)$] If $d^{*} \equiv 1 \pmod{2}$, then 
$$
D_{4}(\bm{b}) D_{4}(\bm{c}) \in \left\{ (8 k - 3) (8 l - 3) \mid k \in \mathbb{Z}, \: 8 l - 3 \in P, \: k \equiv l \: \: ({\rm mod} \: 2) \right\}. 
$$
\end{enumerate}
\end{lem}

Lemma~$\ref{lem:3.5}$ is immediately obtained from \cite[Lemma~4.4]{https://doi.org/10.48550/arxiv.2211.01597} and Remarks~$\ref{rem:3.2}$ and $\ref{rem:3.3}$. 

\begin{lem}\label{lem:3.6}
If $b_{0} + b_{2} \not\equiv b_{1} + b_{3} \pmod{2}$, 
then $F(\bm{d})^{2} \equiv 1 - 8 d^{*} - 8 d \pmod{16}$. 
\end{lem}
\begin{proof}
Let $b_{0} + b_{2} \not\equiv b_{1} + b_{3} \pmod{2}$. 
Then 
$$
d_{0} + d_{2} + d_{4} + d_{6} \not\equiv d_{1} + d_{3} + d_{5} + d_{7} \pmod{2}. 
$$
Therefore, from $n_{0} \equiv n_{1} \equiv 1 \pmod{2}$, 
we have $n_{0}^{2} \equiv n_{1}^{2} \equiv 1 \pmod{8}$. 
Also, from 
\begin{align*}
n_{0} + n_{1} &= 2 (d_{0} + d_{2} + d_{4} + d_{6})^{2} - 4 (d_{0} + d_{2}) (d_{4} + d_{6}) - 4 d^{*}, \\ 
n_{0} - n_{1} &= 2 (d_{1} + d_{3} + d_{5} + d_{7})^{2} - 4 (d_{1} + d_{3}) (d_{5} + d_{7}) - 4 d^{*}, 
\end{align*}
we have 
\begin{align*}
n_{0}^{2} - n_{1}^{2} 
&= (n_{0} + n_{1}) (n_{0} - n_{1}) \\ 
&\equiv 8 (d_{0} + d_{2} + d_{4} + d_{6})^{2} (d_{1} + d_{3}) (d_{5} + d_{7}) \\ 
&\qquad + 8 (d_{1} + d_{3} + d_{5} + d_{7})^{2} (d_{0} + d_{2}) (d_{4} + d_{6}) \\ 
&\qquad \quad + 8 \left\{ (d_{0} + d_{2} + d_{4} + d_{6})^{2} + (d_{1} + d_{3} + d_{5} + d_{7})^{2} \right\} d^{*} \\ 
&\equiv 8 (d_{0} + d_{2} + d_{4} + d_{6}) (d_{1} + d_{3}) (d_{5} + d_{7}) \\ 
&\qquad + 8 (d_{1} + d_{3} + d_{5} + d_{7}) (d_{0} + d_{2}) (d_{4} + d_{6}) + 8 d^{*} \\ 
&\equiv 8 \left\{ (d_{0} + d_{2} + d_{4} + d_{6}) + (d_{1} + d_{3} + d_{5} + d_{7}) \right\} d \\ 
&\qquad - 8 (d_{0} + d_{2} + d_{4} + d_{6}) (d_{0} + d_{2}) (d_{4} + d_{6}) \\ 
&\qquad \quad - 8 (d_{1} + d_{3} + d_{5} + d_{7}) (d_{1} + d_{3}) (d_{5} + d_{7}) + 8 d^{*} \\ 
&\equiv 8 d + 8 d^{*} \pmod{16}. 
\end{align*}
From the above, 
$F(\bm{d})^{2} = n_{0}^{2} n_{1}^{2} \equiv n_{0}^{2} (n_{0}^{2} - 8 d - 8 d^{*}) \equiv 1 - 8 d - 8 d^{*} \pmod{16}$. 
\end{proof}

\begin{lem}\label{lem:3.7}
Suppose that $b_{0} + b_{2} \not\equiv b_{1} + b_{3}$, $d \equiv 1 \pmod{2}$. 
Then we have the following: 
\begin{enumerate}
\item[$(1)$] If $d^{*} \equiv 0 \pmod{2}$, then 
$$
F(\bm{d})^{2} \in \left\{ (4 m - 1)^{2} (8 n + 1) \mid m, n \in \mathbb{Z}, \: m \geq 1, \: m \not\equiv n \: \: ({\rm mod} \: 2) \right\}; 
$$
\item[$(2)$] If $d^{*} \equiv 1 \pmod{2}$, then 
$$
F(\bm{d})^{2} \in \left\{ (4 m - 1)^{2} (8 n + 1) \mid m, n \in \mathbb{Z}, \: m \geq 1, \: m \equiv n \: \: ({\rm mod} \: 2) \right\}. 
$$
\end{enumerate}
\end{lem}
\begin{proof}
From $b_{0} + b_{2} \not\equiv b_{1} + b_{3}$, $d \equiv 1 \pmod{2}$, 
exactly one of $d_{0} + d_{2}$, $d_{4} + d_{6}$, $d_{1} + d_{3}$, $d_{5} + d_{7}$ is even. 
Therefore, we have $n_{0} \equiv - 1 \pmod{4}$ and $n_{1} \equiv 1 \pmod{2}$. 
From $n_{0} \geq 0$, there exists $m \geq 1$ satisfying $n_{0} = 4 m - 1$. 
Also, from $n_{1}^{2} \equiv 1 \pmod{8}$, there exists $n \in \mathbb{Z}$ satisfying $n_{1}^{2} = 8 n + 1$. 
Therefore, $F(\bm{d})^{2} = n_{0}^{2} n_{1}^{2} = (4 m - 1)^{2} (8 n + 1) \equiv 8 (m + n) + 1 \pmod{16}$. 
On the other hand, from Lemma~$\ref{lem:3.6}$, 
we have $F(\bm{d})^{2} \equiv 1 - 8 d^{*} - 8 d \equiv 9 - 8 d^{*} \pmod{16}$. 
From the above, the lemma is proved. 
\end{proof}

\begin{proof}[Proof of Lemma~$\ref{lem:3.1}$]
Suppose that $D_{G}(\bm{a}) = D_{4}(\bm{b}) D_{4}(\bm{c}) F(\bm{d})^{2} \in \mathbb{Z}_{\rm odd}$. 
Then $D_{4}(\bm{b})$ is odd. 
Thus, from Lemma~$\ref{lem:2.4}$ (1), $b_{0} + b_{2} \not\equiv b_{1} + b_{3} \pmod{2}$. 
Therefore, from Lemmas~$\ref{lem:3.4}$ and $\ref{lem:3.6}$, 
\begin{align*}
D_{4}(\bm{b}) D_{4}(\bm{c}) F(\bm{d})^{2} \equiv (1 - 8 d^{*}) (1 - 8 d^{*} - 8 d) \equiv 1 - 8 d \pmod{16}. 
\end{align*}
From this, if $d \equiv 0 \pmod{2}$, then $D_{G}(\bm{a}) \in \{ 16 m + 1 \mid m \in \mathbb{Z} \}$. 
On the other hand, if $d \equiv 1 \pmod{2}$, 
then $D_{G}(\bm{a}) \in A'$ from Lemmas~$\ref{lem:3.5}$ and $\ref{lem:3.7}$, 
where 
\begin{align*}
A' &:= \{ (8 k - 3) (8 l - 3) (4 m - 1)^{2} (8 n + 1) \mid k, m, n \in \mathbb{Z}, \: m \geq 1, \\ 
&\qquad \quad 8 l - 3 \in P, \: k + l \equiv m + n \: \: ({\rm mod} \: 2) \}
\end{align*}
Since $A = A'$ holds, the lemma is proved. 
\end{proof}

\section{Impossible even numbers}
In this section, 
we consider impossible even numbers. 
Let 
$$
P' := \left\{ p \mid p \equiv - 1 \: \: ({\rm mod} \: 4) \: \: \text{is a prime number} \right\}. 
$$

\begin{lem}\label{lem:4.1}
The following hold: 
\begin{enumerate}
\item[$(1)$] $S({\rm C}_{4} \rtimes {\rm C}_{4}) \cap 2 \mathbb{Z} \subset 2^{14} \mathbb{Z}$; 
\item[$(2)$] $S({\rm C}_{4} \rtimes {\rm C}_{4}) \cap 2^{14} \mathbb{Z}_{\rm odd} \subset \left\{ 2^{14} p (2 m + 1), \: 2^{14} q^{2} (2 m + 1) \mid m \in \mathbb{Z}, \: p \in P, \: q \in P' \right\}$. 
\end{enumerate}
\end{lem}

To prove Lemma~$\ref{lem:4.1}$, 
we use the following three lemmas. 

\begin{lem}\label{lem:4.2}
The following hold: 
\begin{enumerate}
\item[$(1)$] If $b_{0} + b_{2} \equiv b_{1} + b_{3} \equiv 0 \pmod{2}$, then 
\begin{align*}
D_{4}(\bm{b}) D_{4}(\bm{c}) \in 
\begin{cases}
2^{11} \mathbb{Z}, & b_{0} + b_{1} + b_{2} + b_{3} \not\equiv c_{0} + c_{1} + c_{2} + c_{3} \pmod{4}, \\ 
2^{8} \mathbb{Z}, & b_{0} + b_{1} + b_{2} + b_{3} \equiv c_{0} + c_{1} + c_{2} + c_{3} \pmod{4}; 
\end{cases}
\end{align*}
\item[$(2)$] If $b_{0} + b_{2} \equiv b_{1} + b_{3} \equiv 1 \pmod{2}$, then 
\begin{align*}
D_{4}(\bm{b}) D_{4}(\bm{c}) \in 
\begin{cases}
2^{8} \mathbb{Z}_{\rm odd}, & (b_{0} + b_{2}) (b_{1} + b_{3}) \equiv \pm 3, \: (c_{0} + c_{2}) (c_{1} + c_{3}) \equiv \pm 3 \pmod{8}, \\ 
2^{9} \mathbb{Z}, & {\text otherwise}. 
\end{cases}
\end{align*}
\end{enumerate}
\end{lem}

Lemma~$\ref{lem:4.2}$ is immediately obtained from \cite[Lemma~5.2]{https://doi.org/10.48550/arxiv.2211.01597}. 

\begin{lem}\label{lem:4.3}
The following hold: 
\begin{enumerate}
\item[$(1)$] If $b_{0} + b_{2} \equiv b_{1} + b_{3} \equiv 0 \pmod{2}$, then 
\begin{align*}
F(\bm{d}) \in \begin{cases} 
2^{2} \mathbb{Z}_{\rm odd}, & b_{0} + b_{1} + b_{2} + b_{3} \not\equiv c_{0} + c_{1} + c_{2} + c_{3} \pmod{4}, \\ 
2^{4} \mathbb{Z}, & b_{0} + b_{1} + b_{2} + b_{3} \equiv c_{0} + c_{1} + c_{2} + c_{3} \pmod{4}; 
\end{cases}
\end{align*}
\item[$(2)$] If $b_{0} + b_{2} \equiv b_{1} + b_{3} \equiv 1 \pmod{2}$, then 
\begin{align*}
F(\bm{d}) \in \begin{cases}
2^{3} \mathbb{Z}_{\rm odd}, & d \equiv 2 \pmod{4}, \\ 
2^{4} \mathbb{Z}, & d \equiv 0 \pmod{4}, 
\end{cases}
\end{align*}
where $d := (d_{0} + d_{2}) (d_{4} + d_{6}) + (d_{1} + d_{3}) (d_{5} + d_{7})$. 
\end{enumerate}
\end{lem}
\begin{proof}
We prove (1). 
Let $b_{0} + b_{2} \equiv b_{1} + b_{3} \equiv 0 \pmod{2}$. 
Then, 
$$
d_{0} + d_{2} + d_{4} + d_{6} \equiv d_{1} + d_{3} + d_{5} + d_{7} \equiv 0 \pmod{2}. 
$$
Also, we have 
\begin{align*}
&b_{0} + b_{1} + b_{2} + b_{3} \equiv c_{0} + c_{1} + c_{2} + c_{3} \pmod{4} \\ 
&\qquad \iff b_{0} + b_{2} + c_{0} + c_{2} \equiv b_{1} + b_{3} + c_{1} + c_{3} \pmod{4} \\ 
&\qquad \iff d_{0} + d_{2} \equiv d_{1} + d_{3} \pmod{2}. 
\end{align*}
Therefore, 
if $b_{0} + b_{1} + b_{2} + b_{3} \not\equiv c_{0} + c_{1} + c_{2} + c_{3} \pmod{4}$, 
then $d_{0} + d_{2} \equiv d_{4} + d_{6} \not\equiv d_{1} + d_{3} \equiv d_{5} + d_{7} \pmod{2}$. 
From this, $n_{0} \equiv n_{1} \equiv 2 \pmod{4}$ holds. 
Thus we have $F(\bm{d}) = n_{0} n_{1} \in 2^{2} \mathbb{Z}_{\rm odd}$. 
On the other hand, 
if $b_{0} + b_{1} + b_{2} + b_{3} \equiv c_{0} + c_{1} + c_{2} + c_{3} \pmod{4}$, 
then $d_{0} + d_{2} \equiv d_{4} + d_{6} \equiv d_{1} + d_{3} \equiv d_{5} + d_{7} \pmod{2}$. 
From this, $n_{0} \equiv n_{1} \equiv 0 \pmod{4}$ holds. 
Thus we have $F(\bm{d}) = n_{0} n_{1} \in 2^{4} \mathbb{Z}$. 
We prove (2). 
Let $b_{0} + b_{2} \equiv b_{1} + b_{3} \equiv 1 \pmod{2}$. 
Then, 
$$
d_{0} + d_{2} + d_{4} + d_{6} \equiv d_{1} + d_{3} + d_{5} + d_{7} \equiv 1 \pmod{2}. 
$$
Thus we have $n_{0} \equiv 2 \pmod{4}$. 
Also, we have 
\begin{align*}
n_{1} 
&= (d_{0} + d_{2} + d_{4} + d_{6})^{2} - (d_{1} + d_{3} + d_{5} + d_{7})^{2} - 2 (d_{0} + d_{2}) (d_{4} + d_{6}) + 2 (d_{1} + d_{3}) (d_{5} + d_{7}) \\ 
&\equiv 2 d \pmod{8}. 
\end{align*}
Therefore, if $d \equiv 2 \pmod{4}$, then $F(\bm{d}) = n_{0} n_{1} \in 2^{3} \mathbb{Z}_{\rm odd}$, 
and if $d \equiv 0 \pmod{4}$, then $F(\bm{d}) = n_{0} n_{1} \in 2^{4} \mathbb{Z}$. 
\end{proof}

\begin{lem}\label{lem:4.4}
If $D_{G}(\bm{a}) \in 2^{14} \mathbb{Z}_{\rm odd}$, 
then 
$$
D_{G}(\bm{a}) \in \{ 2^{14} p (2 m + 1), \: 2^{14} q^{2} (2 m + 1) \mid m \in \mathbb{Z}, \: p \in P, \: q \in P' \}. 
$$
\end{lem}
\begin{proof}
From Lemmas~$\ref{lem:4.2}$ and $\ref{lem:4.3}$, we have 
\begin{align*}
b_{0} + b_{2} \equiv b_{1} + b_{3} \equiv 1 \pmod{2}, \quad D_{4}(\bm{b}) D_{4}(\bm{c}) \in 2^{8} \mathbb{Z}_{\rm odd}, \quad F(\bm{d}) \in 2^{3} \mathbb{Z}_{\rm odd}. 
\end{align*}
That is, 
\begin{align*}
(b_{0} + b_{2}) (b_{1} + b_{3}) \equiv 1 \pmod{2}, \quad (c_{0} + c_{2}) (c_{1} + c_{3}) \equiv \pm 3 \pmod{8}, \quad d \equiv 2 \pmod{4}, 
\end{align*}
where $d := (d_{0} + d_{2}) (d_{4} + d_{6}) + (d_{1} + d_{3}) (d_{5} + d_{7})$. 
We divide the proof into the following cases: 
\begin{enumerate}
\item[(i)] $(b_{0} b_{3} + b_{2} b_{1}, c_{0} c_{3} + c_{2} c_{1}) \equiv (0, 0), \: \: (0, 2) \: \: \text{or} \: \: (2, 0) \pmod{4}$; 
\item[(ii)] $(b_{0} b_{3} + b_{2} b_{1}, c_{0} c_{3} + c_{2} c_{1}) \equiv (2, 2) \pmod{4}$; 
\item[(iii)] $(b_{0} b_{1} + b_{2} b_{3}, c_{0} c_{1} + c_{2} c_{3}) \equiv (0, 0), \: \: (0, 2) \: \: \text{or} \: \: (2, 0) \pmod{4}$; 
\item[(iv)] $(b_{0} b_{1} + b_{2} b_{3}, c_{0} c_{1} + c_{2} c_{3}) \equiv (2, 2) \pmod{4}$. 
\end{enumerate}
First, we consider the case (i). 
Suppose that $b_{0} b_{3} + b_{2} b_{1} \equiv 0 \pmod{4}$. 
Then 
$$
(b_{0} - b_{2}) (b_{1} - b_{3}) = (b_{0} + b_{2}) (b_{1} + b_{3}) - 2 (b_{0} b_{3} + b_{2} b_{1}) \equiv \pm 3 \pmod{8}. 
$$
Thus, $(b_{0} - b_{2})^{2} + (b_{1} - b_{3})^{2} \equiv 10 \pmod{16}$. 
It implies that $(b_{0} - b_{2})^{2} + (b_{1} - b_{3})^{2}$ has at least one prime factor of the form $8 k - 3$. 
That is, $D_{4}(\bm{b})$ has at least one prime factor of the form $8 k - 3$. 
In the same way, we can prove that $D_{4}(\bm{c})$ has at least one prime factor of the form $8 k - 3$ when $c_{0} c_{3} + c_{2} c_{1} \equiv 0 \pmod{4}$. 
From the above, in the case~(i), it holds that $D_{4}(\bm{b}) D_{4}(\bm{c})$ has at least one prime factor of the form $8 k - 3$. 
We can obtain the same conclusion for the case~(iii). 
Next, we consider the case~(ii). 
We prove that $F(\bm{d})$ has at least one prime factor of the form $4 k - 1$. 
From $b_{0} + b_{2} \equiv b_{1} + b_{3} \equiv 1 \pmod{2}$, 
we have $d_{0} + d_{2} + d_{4} + d_{6} \equiv d_{1} + d_{3} + d_{5} + d_{7} \equiv 1 \pmod{2}$. 
Also, since $F(\bm{d})$ is invariant under the each replacing $(d_{0}, d_{2}) \leftrightarrow (d_{4}, d_{6})$ and $(d_{1}, d_{3}) \leftrightarrow (d_{5}, d_{7})$, 
we may assume without loss of generality that $d_{0} + d_{2} \equiv d_{1} + d_{3} \equiv 0$, $d_{4} + d_{6} \equiv d_{5} + d_{7} \equiv 1 \pmod{2}$. 
Then, we have 
\begin{align*}
(d_{0} - d_{2}) + (d_{1} - d_{3}) 
&\equiv (d_{0} - d_{2}) (d_{5} - d_{7}) + (d_{1} - d_{3}) (d_{4} - d_{6}) \\ 
&\equiv (d_{0} + d_{2}) (d_{5} + d_{7}) + (d_{1} + d_{3}) (d_{4} + d_{6}) \\ 
&\qquad + 2 (d_{0} d_{7} + d_{2} d_{5} + d_{4} d_{3} + d_{6} d_{1}) \\ 
&\equiv d + (b_{0} b_{3} + b_{2} b_{1}) - (c_{0} c_{3} + c_{2} c_{1}) \\ 
&\equiv 2 \pmod{4}
\end{align*}
from \cite[Lemma~2.10 (2)]{https://doi.org/10.48550/arxiv.2211.01597}. 
Therefore, $n_{0} \equiv 6 \pmod{8}$ holds. 
From this and $n_{0} \neq - 2$, 
it holds that $n_{0}$ has at least one prime factor of the form $4 k - 1$. 
We can obtain the same conclusion for the case~(iv). 
\end{proof}

\begin{proof}[Proof of Lemma~$\ref{lem:4.1}$]
Suppose that $D_{G}(\bm{a}) = D_{4}(\bm{b}) D_{4}(\bm{c}) F(\bm{d})^{2} \in 2 \mathbb{Z}$. 
Then, from Lemmas~$\ref{lem:2.4}$ (1) and $\ref{lem:2.5}$, 
we have $b_{0} + b_{2} \equiv b_{1} + b_{3} \pmod{2}$. 
Therefore, from Lemmas~$\ref{lem:4.2}$ and $\ref{lem:4.3}$, we obtain (1). 
Also, from Lemma~$\ref{lem:4.4}$, we obtain (2). 
\end{proof}

\section{Possible integers}

In this section, 
we determine all possible integers. 
Lemmas~$\ref{lem:3.1}$ and $\ref{lem:4.1}$ imply that $S \left( {\rm C}_{4} \rtimes {\rm C}_{4} \right)$ does not include every integer that is not mentioned in the following Lemmas~$\ref{lem:5.1}$--$\ref{lem:5.4}$. 

\begin{lem}\label{lem:5.1}
For any $m \in \mathbb{Z}$, 
the following are elements of $S({\rm C}_{4}\rtimes {\rm C}_{4})$: 
\begin{enumerate}
\item[$(1)$] $16 m + 1$; 
\item[$(2)$] $2^{15} (2 m + 1)$; 
\item[$(3)$] $2^{16} m$. 
\end{enumerate}
\end{lem}

\begin{lem}\label{lem:5.2}
For any $k, m \in \mathbb{Z}$ with $m \geq 1$ and $16 l - 3, 16 l + 5 \in P$, 
the following are elements of $S({\rm C}_{4}\rtimes {\rm C}_{4})$: 
\begin{enumerate}
\item[$(1)$] $(16 k - 3) (16 l - 3) (8 m - 1)^{2}$; 
\item[$(2)$] $(16 k + 5) (16 l + 5) (8 m - 1)^{2}$; 
\item[$(3)$] $(16 k + 5) (16 l - 3) (8 m + 3)^{2}$; 
\item[$(4)$] $(16 k - 3) (16 l + 5) (8 m + 3)^{2}$. 
\end{enumerate}
\end{lem}

\begin{lem}\label{lem:5.3}
For any $m \in \mathbb{Z}$ and $p \in P$, 
we have $2^{14} p (2 m + 1) \in S({\rm C}_{4}\rtimes {\rm C}_{4})$. 
\end{lem}

\begin{lem}\label{lem:5.4}
For any $m \in \mathbb{Z}$ and $q \in P'$, 
we have $2^{14} q^{2} (2 m + 1) \in S({\rm C}_{4}\rtimes {\rm C}_{4})$. 
\end{lem}

\begin{proof}[Proof of Lemma~$\ref{lem:5.1}$]
We obtain (1) from 
\begin{align*}
D_{G}(m + 1, m, m, \ldots, m) 
&= D_{4 \times 2}(2 m + 1, 2 m, 2 m, \ldots, 2 m) F(1, 0, 0, \ldots, 0)^{2} \\ 
&= D_{4}(4 m + 1, 4 m, 4 m, 4 m) D_{4}(1, 0, 0, 0) f_{0}(1, 0, 0, 0)^{2} f_{1}(1, 0, 0, 0)^{2} \\ 
&= (8 m + 1)^{2} - (8 m)^{2} \\ 
&= 16 m + 1. 
\end{align*}
We obtain (2) from 
\begin{align*}
&D_{G}(m + 1, m + 1, m + 1, m + 1, m, m + 1, m + 1, m, m + 1, m, m + 1, m, m, m, m, m) \\ 
&\quad = D_{4 \times 2}(2 m + 2, 2 m + 1, 2 m + 2, 2 m + 1, 2 m, 2 m + 1, 2 m + 1, 2 m) F(0, 1, 0, 1, 0, 1, 1, 0)^{2} \\ 
&\quad = D_{4}(4 m + 2, 4 m + 2, 4 m + 3, 4 m + 1) D_{4}(2, 0, 1, 1) f_{0}(0, - 1, 0, 1)^{2} f_{1}(0, 1, 2, 1)^{2} \\ 
&\quad = 2 \left\{ (8 m + 5)^{2} - (8 m + 3)^{2} \right\} \cdot 16 \cdot 2^{2} \cdot (- 4)^{2} \\ 
&\quad = 2^{15} (2 m + 1). 
\end{align*}
We obtain (3) from 
\begin{align*}
&D_{G}(m + 1, m, m + 1, m, m, m, m, m, m, m, m, m, m - 1, m , m, m - 1) \\ 
&\quad = D_{4 \times 2}(2 m + 1, 2 m, 2 m + 1, 2 m, 2 m - 1, 2 m, 2 m, 2 m - 1) F(1, 0, 1, 0, 1, 0, 0, 1)^{2} \\ 
&\quad = D_{4}(4 m, 4 m, 4 m + 1, 4 m - 1) D_{4}(2, 0, 1, 1) f_{0}(0, 1, 0, - 1)^{2} f_{1}(2, 1, 0, 1)^{2} \\ 
&\quad = 2 \left\{ (8 m + 1)^{2} - (8 m - 1)^{2} \right\} \cdot 16 \cdot 2^{2} \cdot 4^{2} \\ 
&\quad = 2^{16} m. 
\end{align*}
\end{proof}

\begin{proof}[Proof of Lemma~$\ref{lem:5.2}$]
Note that the following hold: 
\begin{enumerate}
\item[(i)] For any $16 l - 3 \in P$, there exist $r, s \in \mathbb{Z}$ satisfying $16 l - 3 = (8 r + 3)^{2} + (8 s + 2)^{2}$; 
\item[(ii)] For any $16 l + 5 \in P$, there exist $r, s \in \mathbb{Z}$ satisfying $16 l + 5 = (8 r + 1)^{2} + (8 s + 2)^{2}$; 
\item[(iii)] For any $m \geq 1$, there exist $t_, u, v, w \in \mathbb{Z}$ satisfying 
$$
8 m - 1 = (4 t + 1)^{2} + (4 u + 1)^{2} + (4 v + 1)^{2} + (4 w + 2)^{2}; 
$$
\item[(iv)] For any $m \geq 1$, there exist $t, u, v, w \in \mathbb{Z}$ satisfying 
$$
8 m + 3 = (4 t + 1)^{2} + (4 u + 1)^{2} + (4 v + 1)^{2} + (4 w)^{2}. 
$$
\end{enumerate}
Let 
\begin{align*}
a_{0} &= k + r + w + 1, & a_{1} &= k + s + u, & a_{2} &= k - r - w - 1, & a_{3} &= k - s - u - 1, \\ 
a_{4} &= k + r + v, & a_{5} &= k + s + t, & a_{6} &= k - r - v - 1, & a_{7} &= k - s - t - 1, \\ 
a_{8} &= k + r - w, & a_{9} &= k + s - u, & a_{10} &= k - r + w, & a_{11} &= k - s + u, \\ 
a_{12} &= k + r - v, & a_{13} &= k + s - t, & a_{14} &= k - r + v, & a_{15} &= k - s + t. 
\end{align*}
Then we obtain (1) from (i), (iii) and 
\begin{align*}
D_{G}(\bm{a}) = (16 k - 3) \left\{ (8 r + 3)^{2} + (8 s + 2)^{2} \right\} \left\{ (4 w + 2)^{2} + (4 v + 1)^{2} + (4 u + 1)^{2} + (4t + 1)^{2} \right\}^{2}. 
\end{align*}
Let 
\begin{align*}
a_{0} &= k - r + w + 1, & a_{1} &= k - s + u, & a_{2} &= k + r - w, & a_{3} &= k + s - u, \\ 
a_{4} &= k - r + v, & a_{5} &= k - s + t, & a_{6} &= k + r - v, & a_{7} &= k + s - t, \\ 
a_{8} &= k - r - w, & a_{9} &= k - s - u, & a_{10} &= k + r + w + 1, & a_{11} &= k + s + u + 1, \\ 
a_{12} &= k - r - v, & a_{13} &= k - s - t, & a_{14} &= k + r + v + 1, & a_{15} &= k + s + t + 1. 
\end{align*}
Then we obtain (2) from (ii), (iii) and 
\begin{align*}
D_{G}(\bm{a}) = (16 k + 5) \left\{ (8 r + 1)^{2} + (8 s + 2)^{2} \right\} \left\{ (4 w + 2)^{2} + (4 v + 1)^{2} + (4 u + 1)^{2} + (4t + 1)^{2} \right\}^{2}. 
\end{align*}
Let 
\begin{align*}
a_{0} &= k + r + w + 1, & a_{1} &= k + s + u + 1, & a_{2} &= k - r - w, & a_{3} &= k - s - u, \\ 
a_{4} &= k + r + v + 1, & a_{5} &= k + s + t + 1, & a_{6} &= k - r - v, & a_{7} &= k - s - t, \\ 
a_{8} &= k + r - w + 1, & a_{9} &= k + s - u, & a_{10} &= k - r + w, & a_{11} &= k - s + u, \\ 
a_{12} &= k + r - v, & a_{13} &= k + s - t, & a_{14} &= k - r + v, & a_{15} &= k - s + t. 
\end{align*}
Then we obtain (3) from (i), (iv) and 
\begin{align*}
D_{G}(\bm{a}) &= (16 k + 5) \left\{ (8 r + 3)^{2} + (8 s + 2)^{2} \right\} \left\{ (4 w)^{2} + (4 v + 1)^{2} + (4 u + 1)^{2} + (4t + 1)^{2} \right\}^{2}. 
\end{align*}
Let 
\begin{align*}
a_{0} &= k - r + w, & a_{1} &= k - s + u, & a_{2} &= k + r - w, & a_{3} &= k + s - u, \\ 
a_{4} &= k - r + v, & a_{5} &= k - s + t, & a_{6} &= k + r - v, & a_{7} &= k + s - t, \\ 
a_{8} &= k - r - w, & a_{9} &= k - s - u - 1, & a_{10} &= k + r + w, & a_{11} &= k + s + u, \\ 
a_{12} &= k - r - v - 1, & a_{13} &= k - s - t - 1, & a_{14} &= k + r + v, & a_{15} &= k + s + t. 
\end{align*}
Then we obtain (4) from (ii), (iv) and 
\begin{align*}
D_{G}(\bm{a}) &= (16 k - 3) \left\{ (8 r + 1)^{2} + (8 s + 2)^{2} \right\} \left\{ (4 w)^{2} + (4 v + 1)^{2} + (4 u + 1)^{2} + (4t + 1)^{2} \right\}^{2}. 
\end{align*}
\end{proof}

\begin{proof}[Proof of Lemma~$\ref{lem:5.3}$]
For any $p \in P$, there exist $k, l \in \mathbb{Z}$ satisfying $2 p = (8 k + 3)^{2} + (8 l + 1)^{2}$. 
Let 
\begin{align*}
a_{0} &= m + k + 1, & a_{1} &= m + l, & a_{2} &= m - k, & a_{3} &= m - l, \\ 
a_{4} &= m - k + 1, & a_{5} &= m - l, & a_{6} &= m + k + 1, & a_{7} &= m + l + 1, \\ 
a_{8} &= m + k, & a_{9} &= m + l, & a_{10} &= m - k - 1, & a_{11} &= m - l, \\ 
a_{12} &= m - k, & a_{13} &= m - l, & a_{14} &= m + k + 1, & a_{15} &= m + l. 
\end{align*}
Then we have 
\begin{align*}
D_{G}(\bm{a}) = 2 \left\{ (8 m + 3)^{2} - (8 m + 1)^{2} \right\} \cdot 8 \left\{ (8 k + 3)^{2} + (8 l + 1)^{2} \right\} \cdot 2^{2} \cdot 4^{2} = 2^{14} p (4 m + 1). 
\end{align*}
Let 
\begin{align*}
a_{0} &= m + k, & a_{1} &= m + l, & a_{2} &= m - k - 1, & a_{3} &= m - l, \\ 
a_{4} &= m - k, & a_{5} &= m - l, & a_{6} &= m + k + 1, & a_{7} &= m + l, \\ 
a_{8} &= m + k, & a_{9} &= m + l - 1, & a_{10} &= m - k - 1, & a_{11} &= m - l - 1, \\ 
a_{12} &= m - k, & a_{13} &= m - l - 1, & a_{14} &= m + k, & a_{15} &= m + l. 
\end{align*}
Then we have 
\begin{align*}
D_{G}(\bm{a}) = 2 \left\{ (8 m - 1)^{2} - (8 m - 3)^{2} \right\} \cdot 8 \left\{ (8 k + 3)^{2} + (8 l + 1)^{2} \right\} \cdot 2^{2} \cdot (- 4)^{2} = 2^{14} p (4 m - 1). 
\end{align*}
\end{proof}

\begin{proof}[Proof of Lemma~$\ref{lem:5.4}$]
For any $q \in P'$, there exist $r, s, t, u \in \mathbb{Z}$ satisfying $2 q = (4 r + 1)^{2} + (4 s)^{2} + (4 t + 1)^{2} + (4 u + 2)^{2}$. 
Let 
\begin{align*}
a_{0} &= m + r + 1, & a_{1} &= m + t, & a_{2} &= m - r, & a_{3} &= m - t, \\ 
a_{4} &= m + s, & a_{5} &= m + u + 1, & a_{6} &= m - s, & a_{7} &= m - u, \\ 
a_{8} &= m - r + 1, & a_{9} &= m - t, & a_{10} &= m + r + 1, & a_{11} &= m + t + 1, \\ 
a_{12} &= m - s, & a_{13} &= m - u - 1, & a_{14} &= m + s, & a_{15} &= m + u. 
\end{align*}
Then we have 
\begin{align*}
D_{G}(\bm{a}) = 2 \left\{ (8 m + 3)^{2} - (8 m + 1)^{2} \right\} \cdot 16 \cdot (2 q)^{2} \cdot (- 4)^{2} = 2^{14} q^{2} (4 m + 1). 
\end{align*}
Let 
\begin{align*}
a_{0} &= m + r + 1, & a_{1} &= m + t, & a_{2} &= m - r, & a_{3} &= m - t, \\ 
a_{4} &= m + s, & a_{5} &= m + u, & a_{6} &= m - s, & a_{7} &= m - u - 1, \\ 
a_{8} &= m - r, & a_{9} &= m - t - 1, & a_{10} &= m + r, & a_{11} &= m + t, \\ 
a_{12} &= m - s - 1, & a_{13} &= m - u - 1, & a_{14} &= m + s - 1, & a_{15} &= m + u. 
\end{align*}
Then we have 
\begin{align*}
D_{G}(\bm{a}) = 2 \left\{ (8 m - 1)^{2} - (8 m - 3)^{2} \right\} \cdot 16 \cdot (2 q)^{2} \cdot 4^{2} = 2^{14} q^{2} (4 m - 1). 
\end{align*}
\end{proof}

From Lemmas~$\ref{lem:3.1}$, $\ref{lem:4.1}$ and $\ref{lem:5.1}$--$\ref{lem:5.4}$, 
Theorem~$\ref{thm:1.1}$ is proved.

\clearpage

\bibliography{reference}
\bibliographystyle{plain}

\medskip
\begin{flushleft}
Faculty of Education, 
University of Miyazaki, 
1-1 Gakuen Kibanadai-nishi, 
Miyazaki 889-2192, 
Japan \\ 
{\it Email address}, Yuka Yamaguchi: y-yamaguchi@cc.miyazaki-u.ac.jp \\ 
{\it Email address}, Naoya Yamaguchi: n-yamaguchi@cc.miyazaki-u.ac.jp 
\end{flushleft}

\end{document}